\documentclass[11pt]{amsart} 
\usepackage[english]{babel}
\usepackage{graphicx,epsfig}
\usepackage{bbm}
\usepackage{cite}
\usepackage{amsthm}
\usepackage{esvect}
\usepackage{amsmath,amssymb,latexsym, amsfonts, amscd, amsthm, xy}
\input{xy}
\xyoption{all}

\usepackage{mathrsfs}
\usepackage{appendix}

\usepackage{mathtools}

\usepackage[margin=1in]{geometry}

\usepackage[dvipsnames]{xcolor}

\makeindex 

\usepackage{hyperref}
\hypersetup{pdftoolbar=true, pdftitle={GKform}, pdffitwindow=true, colorlinks=true, citecolor=blue, filecolor=black, linkcolor=purple, urlcolor=red, hypertexnames=false}

\theoremstyle{plain}

\newtheorem{theorem}{Theorem}[section]

\newtheorem{proposition}[theorem]{Proposition}
\newtheorem{corollary}[theorem]{Corollary}

\newtheorem{assumption}[theorem]{Assumption}
\newtheorem*{assumption*}{Assumption}
\newtheorem{lemma}[theorem]{Lemma}

\theoremstyle{definition}
\newtheorem{definition}{Definition}[section]
\newtheorem{remark}{Remark}[section]

\newtheorem*{goal*}{Goal}
\newtheorem*{problem*}{Comment}

\newtheorem{notation}{Notation}[section]

\def\lra{{\longrightarrow}}
\def\SL{{\rm SL}}

\DeclareMathOperator{\Gr}{Gr}

\DeclareMathOperator{\Pic}{Pic}
\DeclareMathOperator{\Fil}{Fil}

\DeclareMathOperator{\sym}{Sym}

\DeclareMathOperator{\dR}{dR} \DeclareMathOperator{\pr}{pr}

\DeclareMathOperator{\et}{et}

\DeclareMathOperator{\End}{End} \DeclareMathOperator{\CH}{CH}
\DeclareMathOperator{\AJ}{AJ}

\DeclareMathOperator{\ord}{ord}

\DeclareMathOperator{\Chow}{\bf Chow}

\DeclareMathOperator{\im}{Im}

\DeclareMathOperator{\gks}{GKS}

\def\d{\mathrm{d}}

\def\Z{\mathbb{Z}}

\def\Q{\mathbb{Q}}

\def\C{\mathbb{C}}

\def\R{\mathbb{R}}

\def\bdf{\begin{defn}}
\def\edf{\end{defn}}

\def\cH{\mathcal{H}}

\def\cO{\mathcal{O}}

\def\Gal{{\rm Gal}}

\def\corr{{\rm Corr}}

\def\d1{d^{(1)}}

\def\d{\mathbf{d}}

\def\T{\mathbb{T}}

\usepackage{tikz}
\usetikzlibrary{positioning} 
\usepackage{tikz-cd}
\usetikzlibrary{arrows,graphs,decorations.pathmorphing,decorations.markings}

\tikzset{
commutative diagrams/.cd,
arrow style=tikz,
diagrams={>=latex}}

\makeatletter
\let\@wraptoccontribs\wraptoccontribs
\makeatother

\begin{document} 

 \title[Torsion properties of modified diagonal classes]{Torsion properties of modified diagonal classes on triple products of modular curves}
\author{David T.-B. G. Lilienfeldt}
 \address{Einstein Institute of Mathematics, Hebrew University of Jerusalem, Israel}
\email{davidterborchgram.lilienfeldt@mail.huji.ac.il}
\date{\today}

\begin{abstract}
Consider three normalised cuspidal eigenforms of weight $2$ and prime level $p$. Under the assumption that the global root number of the associated triple product $L$-function is $+1$, we prove that the complex Abel--Jacobi image of the modified diagonal cycle of Gross--Kudla--Schoen on the triple product of the modular curve $X_0(p)$ is torsion in the corresponding Hecke isotypic component of the Griffiths intermediate Jacobian. The same result holds with the complex Abel--Jacobi map replaced by its \'etale counterpart. As an application, we deduce torsion properties of Chow--Heegner points associated with modified diagonal cycles on elliptic curves of prime level with split multiplicative reduction. The approach also works in the case of composite square-free level.
\end{abstract}

\maketitle

\section{Introduction}

The study of diagonal cycles on triple products of Shimura curves has its origins in the work of Gross, Kudla, and Schoen \cite{grku, grsc}. They introduced a null-homologous modification of the diagonal embedding of the curve in its triple product, referred to as the modified diagonal cycle, or more commonly today as the Gross--Kudla--Schoen cycle.
Given three cuspidal newforms of weight $2$ and square-free level $N$ such that the sign of the functional equation of the associated triple product $L$-function is $-1$, Gross and Kudla \cite{grku} conjectured that the central value at $s=2$ of the derivative of this $L$-function is given by the Beilinson--Bloch height of the modified diagonal cycle on the triple product of an indefinite Shimura curve determined by the local triple product root numbers. A proof of this conjecture was announced in work of Yuan, Zhang, and Zhang \cite{yzz}, but has yet to be published. The Shimura curve in question is the modular curve $X_0(N)$ precisely when the local triple product root numbers are $+1$ at all finite places.

\subsection{Main results}
In this article, we exhibit certain torsion properties of the modified diagonal class\footnote{This terminology refers to the image of the modified diagonal cycle under the complex (or \'etale) Abel--Jacobi map.} on the triple product of the modular curve $X:=X_0(p)$ of prime level $p$. Our results hold more generally for composite square-free level $N$ (see Section \ref{s:squarefree} below). Since the prime level case already contains all relevant ingredients of the proof, we have chosen to focus on this case.

The modified diagonal cycle depends on a rational base point $e$ of $X$. It will be denoted by $\Delta_{\gks}(e)$ and viewed as an element of the Chow group $\CH^2(X^3)_0$ of null-homologous codimension $2$ algebraic cycles on $X^3$ modulo rational equivalence.
Let $f_1, f_2$, and $f_3$ be three normalised cuspidal eigenforms of weight $2$ and level $\Gamma_0(p)$, and denote by $F:=f_1\otimes f_2\otimes f_3$ their triple product. We place ourselves in the setting where the global root number $W(F)$ of the triple product $L$-function $L(F,s)$ associated with $F$ is $+1$. This assumption forces $L(F, s)$ to vanish to even order at its centre $s=2$. Comparing with the more classical situation of Heegner points studied in the seminal work of Gross and Zagier \cite{GZ}, it seems reasonable to expect that the $F$-isotypic Hecke component $\Delta_{\gks}^F(e)$ of the modified diagonal cycle, with $e\in X(\Q)$, is torsion in the Chow group $\CH^2(X^3)_0(\Q)$, in line with the predictions of the Beilinson--Bloch conjectures \cite{bloch}. While it appears difficult to prove a torsion statement directly in the Chow group, we can prove the corresponding result for the image of $\Delta_{\gks}^F(e)$ under the complex Abel--Jacobi map
 \begin{equation}\label{introCAJ}
 \AJ_{X^3} : \CH^2(X^3)_0(\C) \lra J^2(X^3/\C),
 \end{equation}
 whose target is the Griffiths intermediate Jacobian of $X^3$. The main result is the following.
 
\begin{theorem}\label{main}
Let $f_1, f_2$, and $f_3$ be three normalised eigenforms of weight $2$ and level $\Gamma_0(p)$, denote by $F=f_1\otimes f_2\otimes f_3$ their triple product, and suppose the global root number of $L(F, s)$ is $+1$.
Then $\AJ_{X^3}(\Delta_{\gks}^F(e))$ is torsion in $J^2(X^3/\C)$, for all $e\in X(\Q)$.
\end{theorem}

The kernel of the complex Abel--Jacobi map (restricted to cycles defined over $\bar{\Q}$) is conjectured to be torsion \cite[Conj. 9.12]{jannsen}. Conditional on this conjecture, Theorem \ref{main} implies that $\Delta_{\gks}^F(e)$ is torsion in the Chow group. 
The same statement as in Theorem \ref{main} holds with the complex Abel--Jacobi map replaced by its $\ell$-adic \'etale counterpart \cite{bloch}
\[
\AJ^{\et}_{X^3} : \CH^2(X^3)_0(\Q) \lra H^1(\Q, H^{3}_{\et}(X^3_{\bar{\Q}}, \Q_\ell(2))),
\]
with $\ell$ a rational prime (see Remark \ref{rem:GKStorsion}). 

In the special case $p=37$, using numerical results due to Stein \cite[Appendix]{ddlr}, we deduce the following, where $\xi_\infty$ denotes the cusp of $X$ at infinity. 

\begin{corollary}\label{coromain}
Let $f$ and $g$ be the normalised cuspidal eigenforms of weight $2$ and level $\Gamma_0(37)$ corresponding to the elliptic curves with Cremona labels $37$b and $37$a, and let $F:=g\otimes g\otimes f$. Then $\AJ_{X_0(37)^3}(\Delta_{\gks}^F(\xi_\infty))$ is a non-trivial $6$-torsion element of $J^2(X_0(37)^3/\C)$.
\end{corollary}

\subsection{Application to Chow--Heegner points}

Chow--Heegner points were introduced by Bertolini, Darmon, and Prasanna \cite{bdp2} as a generalisation of the construction of Heegner points. The idea is to produce rational points on elliptic curves by pushing forward algebraic cycles on higher dimensional varieties using suitable correspondences, or generalised modular parametrisations, as they are referred to in \cite{bdp2}. 

Let $f$ be a normalised cuspidal eigenform of weight $2$ and level $\Gamma_0(p)$ with rational Fourier coefficients. Denote by $E_f$ the elliptic curve over $\Q$ of conductor $p$ associated with $f$ by Eichler and Shimura \cite{shimura58}. Using an auxiliary distinct normalised cuspidal eigenform $g$ of weight $2$ and level $\Gamma_0(p)$, it is possible to construct a correspondence $\Pi_{g, f} \in \CH^2(X^3\times E_f)(\Q)$, which gives rise via push-forward to a generalised modular parametrisation
\[
\Pi_{g,f,*} : \CH^2(X^3)_0 \lra \CH^1(E_f)_0=E_f.
\]
The Chow--Heegner point associated with the modified diagonal cycle is then defined as
$$P_{g,f}(e):=\Pi_{g,f,*}(\Delta_{\gks}(e))\in E_f(\Q).$$ 

Darmon, Rotger, and Sols \cite{DRS} have studied such points, in the broader context of Shimura curves over totally real fields, notably by computing their images under the complex Abel--Jacobi map in terms of iterated integrals. Methods have been developed by Darmon, Daub, Lichtenstein, Rotger, and Stein \cite{ddlr} to numerically calculate such points in the case of modular curves. 

Let $F:=g\otimes g\otimes f$. We exhibit a correspondence mapping $\Delta_{\gks}^F(e)$ to $P_{g,f}(e)$. When the global root number $W(F)$ is $-1$, Darmon, Rotger, and Sols \cite{DRS} have studied the non-torsion properties of $P_{g,f}(\xi_\infty)$, building on \cite{yzz}. In the complementary situation when $W(F)=+1$, we use Theorem \ref{main} and functoriality of Abel--Jacobi maps with respect to correspondences to deduce the following: 

\begin{theorem}\label{main2}
Let $f$ and $g$ be as above, and let $F=g\otimes g\otimes f$. If $W(F)=+1$, then the Chow--Heegner point $P_{g,f}(e)$ is torsion in $E_f(\Q)$, for all $e\in X(\Q)$.
\end{theorem}

Theorem \ref{main2} with $e=\xi_\infty$ recovers a result of Daub \cite[Thm. 3.3.8]{daubthesis} by a different method in the case of prime level. 

\subsection{Strategy of the proof}

The key ingredient in the proof of Theorem \ref{main} is the Atkin--Lehner involution $w_p$ of $X$. The global root number of $W(F)$ is the product of the global root numbers of $f_1, f_2$, and $f_3$, which are each equal to the negative of their $w_p$-eigenvalue. As a consequence, the assumption that $W(F)$ equals $+1$ translates into information about the action of $w_p\times w_p\times w_p$ on $F$, and consequently on $\AJ_{X^3}(\Delta_{\gks}^F(e))$, as the latter lies in the $F$-isotypic Hecke component of the intermediate Jacobian by functoriality of the Abel--Jacobi map with respect to correspondences. The work of Mazur \cite{mazur78} provides necessary information about the rational points $X(\Q)$ and the action of $w_p$ on them.

\subsection{Composite square-free level}\label{s:squarefree}

The arguments of this paper carry over to the more general setting where the level $N$ is composite, but square-free. This is the situation initially considered in the work of Gross and Kudla \cite{grku}. It becomes necessary to replace the word eigenform by the word newform throughout the text. 

Let $f_1, f_2, f_3$ be three normalised newforms of weight $2$ and level $\Gamma_0(N)$, and let $F:=f_1\otimes f_2\otimes f_3$.
The level being square-free guarantees that the local root numbers $W_p(F)$ for $p\mid N$ are the products of the local root numbers at $p$ of $f_1, f_2$, and $f_3$, which are each the negative of their $w_p$-eigenvalue. The Atkin--Lehner correspondences $w_p$, $p\mid N$, commute with the good Hecke correspondences $T_n$ (i.e., with $(n, N)=1$), and this is sufficient for our purposes (see Remark \ref{rem:kani}). Assume that there exists $p\mid N$ for which $W_p(F)=-1$. Using multiplicity one for newforms, this assumption can be parlayed into information about the torsion properties of the images of modified diagonal cycles under Abel--Jacobi maps, as long as one has sufficient understanding of the action of the Atkin--Lehner involution $w_p$ on the rational points of $X_0(N)$. The only rational points on composite level modular curves $X_0(N)$ of genus $\geq 2$ are the rational cusps \cite{kenku}. It is known that the subgroup of the Jacobian $J_0(N)$ generated by the cusps is torsion by the Manin--Drinfeld theorem \cite{manin}. It follows that Theorem \ref{main} remains true for normalised newforms $f_1,f_2,f_3$ of composite square-free level under the assumption that $W_p(F)=-1$ for some $p\mid N$. 

The proof of Theorem \ref{main2} adapts verbatim to the setting of composite square-free level, provided that the eigenforms are newforms and $W_p(F)=-1$ for some $p\mid N$.  
This recovers \cite[Thm. 3.3.8]{daubthesis} by a different approach. 

Examining Stein's Table 2 in \cite[Appendix]{ddlr}, we obtain results similar to Corollary \ref{coromain}, e.g., in the following cases:
\begin{itemize}
\item $N=57$: $f$ corresponds to the elliptic curve with Cremona label $57$c and $g$ corresponds to the curves with labels $57$a or $57$b.
\item $N=58$: $f$ corresponds to the elliptic curve with Cremona label $58$b and $g$ corresponds to the curve with label $58$a.
\end{itemize}

\subsection{Related work}

The approach taken in this paper is explicit and elementary, exploiting the connection between triple product root numbers and eigenvalues of Atkin--Lehner involutions. A more powerful approach is considered in the work of Yuan, Zhang, and Zhang \cite{yzz}, using Prasad's dichotomy for the existence of trilinear forms on automorphic representations. Forthcoming work of Qiu and Zhang \cite{qiuzhang} further develops this approach and gives applications.

\subsection{Outline}

Background on cusp forms of weight $2$ is recalled in Section \ref{s:cuspform}. Section \ref{s:triple} recalls facts about the triple product $L$-function and states the Beilinson--Bloch conjecture in this setting. Section \ref{s:AJ} constitutes the proof of Theorem \ref{main}. The application to Chow--Heegner points is given in Section \ref{s:CH}. Corollary \ref{coromain} is discussed in Remark \ref{rem:stein}.

\subsection{Notational conventions}\label{s:notation}

Fix a complex embedding $\bar{\Q}\hookrightarrow \C$, as well as $p$-adic embeddings $\bar{\Q}\hookrightarrow \C_p$ for each rational prime $p$. In this way, all finite extensions of $\Q$ are viewed simultaneously as subfields of $\C$ and $\C_p$. The subscript $\Q$ will denote the tensor product with $\Q$ over $\Z$. Given two varieties $X$ and $Y$, we write $\corr^{r}(X, Y):=\CH^{\dim X+r}(X\times Y)$ for the group of correspondences of degree $r$, where $\CH$ denotes the Chow group. We use the same conventions and notations for motives as in \cite[\S 0]{deligne}, and we denote by $\Chow(\Q)$ the category of pure Chow motives over $\Q$.

\section{Cusp forms}\label{s:cuspform}

Let $p> 3$ be a rational prime. Let $Y:=Y_0(p)$ be the modular curve over $\Q$ for the congruence subgroup $\Gamma_0(p)\subset \SL_2(\Z)$ consisting of matrices, which are upper-triangular modulo $p$. It admits a canonical proper desingularisation $Y_0(p)\hookrightarrow X_0(p)$, obtained over the complex numbers by adjoining the cusps. The curve $X:=X_0(p)$ is a geometrically connected, smooth, and proper curve over $\Q$. It is the coarse moduli scheme representing pairs $(E, H)$ consisting of a generalised elliptic curve $E$ over a $\Q$-scheme $S$, together with a cyclic subgroup scheme $H$ of order $p$. It admits a uniformisation by the extended Poincar\'e upper half-plane 
\begin{equation}\label{uniformisation}
\mathcal{H}^*\lra X(\C), \qquad \tau\mapsto (\C/\Z\oplus \tau\Z, \langle1/p +\Z\oplus \tau\Z\rangle),
\end{equation}
which identifies $X(\C)$ with the quotient $\Gamma_0(p)\backslash\mathcal{H}^*$.
There are two cusps $\xi_\infty$ and $\xi_0$ on $X$, which correspond via \eqref{uniformisation} to the points $i\infty$ and $0$ of $\cH^*$.
The genus $g_X$ of $X$ is given by the formula 
\begin{equation}\label{g0}
g_X=
\begin{cases}
\lfloor \frac{p+1}{12}\rfloor-1, & \text{ if } p\equiv 1 \pmod {12} \\
\lfloor \frac{p+1}{12}\rfloor, & \text{ otherwise.}
\end{cases}
\end{equation}

The space $S_2(\Gamma_0(N))$ of weight $2$ cusp forms of level $\Gamma_0(p)$ is naturally identified with the space of global sections of the sheaf of regular differential $1$-forms on $X$ via the isomorphism
\begin{equation}\label{cuspH}
S_2(\Gamma_0(p))\overset{\sim}{\lra} H^0(X_{\C}, \Omega^1_{X}), \qquad f\mapsto \omega_f:=2\pi i f(\tau)d\tau.
\end{equation}
In particular, the dimension of $S_2(\Gamma_0(p))$ is equal to $g_X$. 

\subsection{Hecke operators}\label{s:hecke}

The curve $X$ is equipped with the usual collection of Hecke correspondences, which act on cohomology and give rise to operators on $S_2(\Gamma_0(p))$ via \eqref{cuspH}. These correspondences and their induced operators are denoted by $U_p$ and $T_n$, for integers $n\geq 1$ coprime to $p$. Their precise definition can be found in \cite[(3.1)]{atkinlehner}. 

The curve $X$ also comes equipped with the Atkin--Lehner involution $w_p$. It is defined, following the moduli description, by mapping a $p$-isogeny $\phi : E\lra E'$ of elliptic curves to its dual isogeny $\phi^\vee : E'\lra E$. In terms of covering spaces, using \eqref{uniformisation}, it is given by $\tau\mapsto -\frac{1}{p\tau}$, where $\tau\in \cH$. This involution is defined over $\Q$ and thus maps $\Q$-rational points of $X$ to $\Q$-rational points. It induces, via \eqref{cuspH}, an operator on $S_2(\Gamma_0(p))$, which we also denote by $w_p$.

The operators $T_m$, with $(m, p)=1$, on $S_2(\Gamma_0(p))$ commute with the operators $T_n$, $U_p$ and $w_p$ by \cite[Lem. 17]{atkinlehner}. 
Let $\T:=\T(p)$ denote the $\Q$-algebra generated by the operators $T_n$, with $(n,p)=1$. The space of cusp forms $S_2(\Gamma_0(p))$ admits a basis of eigenforms for $\T$ \cite[Thm. 2]{atkinlehner}. 

Let $f = \sum_{n\geq 1} a_n(f)q^n \in S_2(\Gamma_0(p))$ be a normalised eigenform, in the sense that $a_1(f)=1$. Because the level is prime, there are no oldforms. As a consequence, we have the equality of operators $U_p=-w_p$. In particular, the operators $U_p$ and $w_p$ commute. Note that this is only the case for general composite level when restricting to newforms \cite[Lem. 17]{atkinlehner}. 
It follows that $w_p(f)=-a_p(f)f$. In particular, we have $a_p(f)\in \{ \pm 1 \}$. 

The normalised eigenform $f$ determines a surjective algebra homomorphism $\lambda_f : \T \lra K_f$ by sending $T_n$ to $a_n(f)$. Here $K_f$ is the totally real finite extension of $\Q$ generated by the Fourier coefficients $a_n(f)$ of $f$.  

Let $S_2(\Gamma_0(p))_f$ denote the $f$-isotypic component of $S_2(\Gamma_0(p))$ consisting of cusp forms $f'$ in $S_2(\Gamma_0(p))$ such that $T(f')=\lambda_f(T) f'$, for all $T\in \T$. By the theorem of multiplicity one \cite[Lem. 20 \& 21]{atkinlehner} of Atkin and Lehner for newforms, the space $S_2(\Gamma_0(p))_f$ is $1$-dimensional over $\C$.
We have the spectral decomposition 
\[
S_2(\Gamma_0(p)) = \bigoplus_{h} S_2(\Gamma_0(p))_h,
\] 
where the sum is taken over all normalised eigenforms $h\in S_2(\Gamma_0(p))$.
Since the dual space $S_2(\Gamma_0(p))^\vee$ is a free $\T_{\C}$-module of rank one by multiplicity one, we similarly obtain a decomposition 
\[
\T_{\C}=\bigoplus_h \T_{\C, h},
\]
where $\T_{\C, h}$ denotes the algebra of Hecke operators $T_n$, with $(n, p)=1$, acting on $S_2(\Gamma_0(p))_h$, which is again a $\C$-vector space of dimension one. 

Let $[f]$ denote the $\Gal(\bar{\Q}/\Q)$ orbit of $f$. Form the $\C$-vector space $\bigoplus_{g\in [f]} S_{2}(\Gamma_0(p))_g$ of dimension $d_f:=[K_f : \Q]$, and consider the $\Q$-subspace $S_2(\Gamma_0(p))_{[f]}$ of forms with rational coefficients. This $\Q$-vector space is stable under the action of $\T_{\Q}$, and we let $\T_{\Q, [f]}$ denote the $\Q$-algebra generated by the Hecke operators acting on $S_2(\Gamma_0(p))_{[f]}$. We then have the decomposition 
\begin{equation}\label{heckedecomp}
\T=\bigoplus_{[h]} \T_{\Q, [h]} \simeq \bigoplus_{[h]} K_h,
\end{equation}
where the sum is taken over all $\Gal(\bar{\Q}/\Q)$ conjugacy classes of normalised eigenforms in $S_2(\Gamma_0(p))$.

Let $\End_\Q(J)$ denote the ring of endomorphisms defined over $\Q$ of the Jacobian $J:=\Pic_{X/\Q}^0$ of $X$, and let $\End_\Q^0(J):=\End_\Q(J)\otimes \Q$. As $p$ is prime, we have $\End_\Q^0(J)= \T$ \cite[Cor. 3.3]{ribet75}. To summarise, we have the decomposition 
\begin{equation}\label{hecke}
\End_\Q^0(J)= \T\simeq \bigoplus_{[h]} K_h.
\end{equation}

It will be useful to remark that there is a natural isomorphism 
\begin{equation}\label{endo}
\End_\Q^0(J)\simeq (\CH^1(X^2)\otimes \Q)/(\pr_1^* \CH^1(X)\otimes \Q+\pr_2^* \CH^1(X)\otimes \Q).
\end{equation}
See for instance \cite[Thm. 11.5.1]{langebirkenhake}.

\begin{remark}\label{rem:kani}
The exposition is simplified by the assumption that the level is prime, which implies that there are no oldforms. We refer to \cite[\S 1.6]{ddt} for the decomposition \eqref{heckedecomp} in the more general setting of composite level $N$. In this case, the algebra $\End^0_\Q(J)$ is a product of matrix algebras. It contains $\T$ as its center and the full Hecke algebra as a maximal commutative subalgebra. Moreover, $\End^0_\Q(J)$ is generated as a $\Q$-algebra by $\T$, together with certain degeneracy operators \cite[Thm. 1]{kani}.
\end{remark}

\subsection{Hecke projectors}\label{s:motivef}

Let $f = \sum_{n\geq 1} a_n(f)q^n \in S_2(\Gamma_0(p))$ be a normalised eigenform. Denote by $V:=S_2(\Gamma_0(p))^\vee$ the $\C$-dual of $S_2(\Gamma_0(p))$. The complex points of the Jacobian $J$ are 
\[
J(\C)=H^0(X, \Omega_{X}^1)^\vee \slash \im H_{1}(X(\C), \Z),
\]
where $\Lambda:=\im H_{1}(X(\C), \Z)$ is viewed as a lattice via integration of differential forms. By \eqref{cuspH}, we thus have an identification $J(\C)=V/\Lambda$ as a $g_X$-dimensional complex torus, where we recall that $g_X$ is the genus of $X$. Let $V_f$ be the subspace of $V$ on which $\T$ acts via the homomorphism $\lambda_f : \T \lra K_f$, and let $\pi_f : V\lra V_f$ be the orthogonal projection with respect to the Petersson scalar product. The projector $\pi_f$ naturally belongs to $\T_{K_f}=\T\otimes_{\Q} K_f$, and by \eqref{hecke} and \eqref{endo} we may view $\pi_f$ as an idempotent correspondence $t_f \in \corr^0(X, X)_{K_f}$. 

\begin{definition}\label{def:motivef}
The motive $M(f):=(X, t_f, 0)\in \Chow(\Q)_{K_f}$ over $\Q$ with coefficients in $K_f$ is the motive of $f$.
\end{definition}

For a prime $\ell$, the $\ell$-adic realisation of $M(f)$ is given by
\begin{equation}\label{repf}
M(f)_\ell=(t_f)_* H^*_{\et}(X_{\bar{\Q}}, \Q_\ell) = H^1_{\et}(X_{\bar{\Q}}, \Q_\ell)_f = V_\ell(f),
\end{equation}
where $V_\ell(f)$ is the $2$-dimensional $\ell$-adic Galois representation associated with $f$ given the fixed field embeddings of Section \ref{s:notation}. 

If we let $V_{[f]}:=\bigoplus_{g\in [f]} V_g$ and $\pi_{[f]}:=\sum_{g\in [f]} \pi_f$, then $\pi_{[f]}$ is the orthogonal projection $V\lra V_{[f]}$ with respect to the Petersson scalar product. The abelian variety $A_{[f]}$ associated with the Galois orbit $[f]$ by Eichler and Shimura \cite{shimura58}, is isomorphic over $\C$ to the complex torus $V_{[f]}/\pi_{[f]}(\Lambda)$ \cite[Lem. 1.46]{ddt}. Moreover, the projection map $\pi_{[f]} : V/\Lambda \lra V_{[f]}/\pi_{[f]}(\Lambda)$ corresponds to the natural quotient $J\lra A_{[f]}$. In particular, $\pi_{[f]}$ naturally belongs to $\T=\End_\Q^0(J)$, and corresponds under \eqref{hecke} to the idempotent element $e_{[f]}\in  \bigoplus_{[h]} K_h$ which has $1$ as $[f]$-coordinate and $0$ as $[h]$-coordinate for $[h]\neq [f]$. By \eqref{endo}, we may view $\pi_{[f]}$ as an idempotent correspondence $t_{[f]} \in \corr^0(X, X)_\Q$. It follows that the motive $M([f]):=(X, t_{[f]}, 0)\in \Chow(\Q)_\Q$ is equal to the modular abelian variety $A_{[f]}$. 

\section{Triple products}\label{s:triple}

Let 
$
f_1=\sum_{n\geq 1} a_n(f_1)q^n, f_2=\sum_{n\geq 1} a_n(f_2)q^n,$ and $f_3=\sum_{n\geq 1} a_n(f_3)q^n
$
be three normalised cuspidal eigenforms of weight $2$ and level $\Gamma_0(p)$, and let $F:=f_1\otimes f_2\otimes f_3$ be the associated cusp form of weight $(2, 2, 2)$ for $\Gamma_0(p)^3$. Let $K_F = K_{f_1}\cdot K_{f_2} \cdot K_{f_3}$ denote the compositum of the Hecke fields of the forms $f_1, f_2,$ and $f_3$. Using the Hecke isotypic projectors of Section \ref{s:motivef}, we define the idempotent correspondence 
\begin{equation}\label{tF}
t_F := t_{f_1} \otimes t_{f_2} \otimes t_{f_3}=\pr_{14}^*(t_{f_1}) \cdot \pr_{25}^*(t_{f_2})\cdot \pr_{36}^*(t_{f_3}) \in \corr^0(X^3, X^3)_{K_F},
\end{equation}
where $\pr_{ij} : X^6 \lra X^2$ denotes the natural projection to the $i$-th and $j$-th components.

\begin{definition}
The motive of the triple product $F$ is defined as the motive $$M(F):=M(f_1)\otimes M(f_2)\otimes M(f_3)=(X_0(p)^3, t_F, 0)\in \Chow(\Q)_{K_F}$$ over $\Q$ with coefficients in $K_F$.
\end{definition}

The $\ell$-adic realisations of $M(F)$ give rise to a compatible family of $8$-dimensional $\ell$-adic Galois representations 
\[
\{ V_\ell(F):=M(F)_\ell=V_\ell(f_1)\otimes V_\ell(f_2)\otimes V_\ell(f_3) \}_\ell,
\]
where the representations $V_\ell(f_{i})$ for $i\in \{ 1,2,3 \}$ are the ones appearing in \eqref{repf}. 

The de Rham realisation is given by 
\[
M(F)_{\dR}=H^1(X_0(p)(\C), \C)_{f_1}\otimes H^1(X_0(p)(\C), \C)_{f_2}\otimes H^1(X_0(p)(\C), \C)_{f_3}.
\]
It follows that the Hodge numbers of $M(F)$ are 
\begin{equation}\label{hnF}
h^{3,0}(M(F))=h^{0, 3}(M(F))=1 \qquad \text{ and } \qquad h^{2,1}(M(F))=h^{1, 2}(M(F))=3.
\end{equation}

\subsection{Triple product $L$-functions}\label{s:tripleLfunction}

Following for instance \cite{del73}, we attach to the motive of $F$ the $L$-function 
\[
\Lambda(M(F)/\Q, s):=\prod_{v} L(\sigma'_{F, v}, s)=L(\sigma'_{F, \infty}, s)\prod_{q} L(\sigma'_{F, q}, s),
\]
where $\sigma'_{F, v}$ is the Weil--Deligne representation attached to $F$ at the place $v$.
This is the triple product $L$-function associated with $f_1, f_2$, and $f_3$ and studied in \cite{grku}.

\begin{notation}
We will alternatively write $\Lambda(F, s)$ or $\Lambda(f_1, f_2, f_3, s)$ for this $L$-function. Similarly, we write $L(F,s)$ or $L(f_1, f_2, f_3, s)$ for the finite part $\prod_{q} L(\sigma'_{F, q}, s)$ and also refer to this as the triple product $L$-function.
\end{notation}

Using explicit descriptions of the Weil--Deligne representations of $F$, one can work out expressions for the local factors at the finite primes \cite[(1.7), (1.8)]{grku}.
At primes $q\neq p$ these factors are of degree $8$, while the factor at $p$ is of degree $3$. 
Following the prescription of \cite{del73} and using \eqref{hnF}, the local $L$-factor at infinity is given by 
\[
L(\sigma'_{F, \infty}, s)=
2^4(2\pi)^{3-4s}\Gamma(s-1)^3\Gamma(s).
\]
Letting
$\Lambda^*(F, s):=(p^5)^{\frac{s}{2}}\Lambda(F, s)$,
general conjectures \cite{del73} predict that this $L$-function admits an analytic continuation to the entire complex plane, and satisfies the functional equation 
\begin{equation}\label{eq:FE}
\Lambda^*(F, s)=W(F) \cdot\Lambda^*(F, 4-s),
\end{equation}
where $W(F)\in \{ \pm 1 \}$ is the global root number of the motive $M(F)$. 
The analytic continuation of $\Lambda^*(F, s)$, as well as the functional equation \eqref{eq:FE}, have been established by Gross and Kudla in the more general setting of square-free composite level \cite[Prop 1.1]{grku}. The global root number, as stated in \cite[\S 1]{grku}, is given by 
\begin{equation}\label{rem:tripleroot}
W(F)=a_p(f_1)a_p(f_2)a_p(f_3).
\end{equation}
The latter can be established by computing the local epsilon factors using explicit descriptions of the Weil--Deligne representations of $M(F)$. This is done in detail for instance in \cite[Prop. 4.5]{mythesis}.

\subsection{The Beilinson--Bloch conjecture}

The centre of symmetry of the functional equation \eqref{eq:FE} is the point $s=2$, at which $L(F, s)$ has no pole. Moreover, $L(\sigma'_{\infty, F}, s)$ has neither zero nor pole at $s=2$, so the centre is a critical point, and we have 
\begin{equation}\label{parity}
W(F)=(-1)^{\ord_{s=2} L(F, s)}.
\end{equation}
The Beilinson--Bloch conjecture \cite{bloch} predicts in this setting that 
\begin{equation}\label{bb}
\ord_{s=2} L(F, s)=\dim_{K_F} \: (t_F)_* (\CH^2(X^3)_0(\Q) \otimes K_F),
\end{equation}
where $\CH^2(X^3)_0$ is the Chow group of null-homologous algebraic cycles of codimension $2$ on $X^3$ modulo rational equivalence.

In the case when $W(F)=+1$, Gross and Kudla proved a formula for the central value $L(F,2)$, expressing it as a product of a complex period and an algebraic number \cite[Prop. 10.8]{grku}. This algebraic number admits an explicit description in terms of the coefficients of the Jacquet--Langlands transfers of $f_1, f_2,$ and $f_3$ to the definite quaternion algebra ramified at $p$ and $\infty$. 

In the case when $W(F)=-1$, the $L$-function $L(F, s)$ vanishes to odd order at its centre $s=2$. By \eqref{bb}, we expect the $F$-isotypic component of $\CH^2(X^3)_0(\Q)\otimes K_F$ to have dimension greater or equal to $1$. 
A natural element of $\CH^2(X^3)_0(\Q)$ to consider is the modified diagonal cycle, also referred to as the Gross--Kudla--Schoen cycle. Let $\Delta$ denote the image of $X$ under the diagonal embedding $X\lra X^3$, i.e., 
\begin{equation}\label{def:delta}
\Delta=\{ (x,x,x) \: \vert \: x\in X \} \subset X^3.
\end{equation}
In order to get a null-homologous cycle, we apply a projector to $\Delta$, which was originally considered by Gross and Kudla \cite{grku}, and defined in full generality by Gross and Schoen \cite{grsc}.

\begin{definition}\label{def:pgks}
Let $C$ be a smooth, projective, and geometrically connected curve over a number field $k$, and let $e$ be a $k$-rational point of $X$. For any non-empty subset $T$ of $\{ 1, 2, 3 \}$, let $T'$ denote the complementary set. Write $p_T : C^3\lra C^{\vert T \vert}$ for the natural projection map and let $q_T(e) : C^{\vert T\vert}\lra C^3$ denote the inclusion obtained by filling in the missing coordinates using the point $e$. Let $P_{T}(e)$ denote the graph of the morphism $q_{T}(e)\circ p_T : C^3 \lra C^3$ viewed as a codimension $3$ cycle on the product $C^3\times C^3$. Define the Gross--Kudla--Schoen projector 
\[
P_{\gks}(e):=\sum_T (-1)^{\vert T'\vert} P_T(e) \in \CH^3(C^3\times C^3),
\]
where the sum is taken over all subsets of $\{ 1, 2, 3 \}$. This is an idempotent in the ring of correspondences of $C^3$ with the property that it annihilates the cohomology groups $H^{i}(C^3(\C), \Z)$, for $i\in \{ 4,5,6 \}$, and maps $H^3(C^3(\C), \Z)$ onto the K\"unneth summand $H^1(C(\C), \Z)^{\otimes 3}$ \cite[Cor. 2.6]{grsc}.
\end{definition}

Given a rational point $e\in X(\Q)$, the Gross--Kudla--Schoen cycle based at $e$ is defined as
\begin{equation}\label{def:gkscycle}
\Delta_{\gks}(e):=P_{\gks}(e)_*(\Delta) \in \CH^2(X^3)_0(\Q).
\end{equation}
Note that $\Delta_{\gks}(e)$ is in fact null-homologous, as $P_{\gks}(e)$ annihilates $H^4(X^3(\C), \Z)$, the target of the cycle class map. 

Gross and Kudla \cite[Conj. 13.2]{grku} conjectured the formula 
\begin{equation}\label{conj:grku}
\frac{L'(F, 2)}{\Omega_F} = \langle (t_F)_* (\Delta_{\gks}(\xi_\infty)), (t_F)_* (\Delta_{\gks}(\xi_\infty)) \rangle^{\rm BB}, 
\end{equation}
where $\langle \: , \: \rangle^{\rm BB} : \CH^2(X^3)_0(\Q)\otimes \R \times \CH^2(X^3)_0(\Q) \otimes \R \lra \R$ denotes the Beilinson--Bloch height pairing \cite[(13.9)]{grku}. A proof of \eqref{conj:grku} due to Yuan, Zhang, and Zhang was announced in \cite{yzz}, but has not yet been published.

\section{Abel--Jacobi maps}\label{s:AJ}

Let $f_1, f_2,$ and $f_3$ be three normalised eigenforms in $S_2(\Gamma_0(p))$, and let $F=f_1\otimes f_2\otimes f_3$. We work under the following assumption on the sign of the functional equation (\ref{eq:FE}). 
\begin{assumption}\label{sign} 
$W(F)=+1. $
\end{assumption}

Under this assumption, by \eqref{parity}, the $L$-function $L(F, s)$ vanishes to even order at the central critical point $s=2$, and we have at our disposal the Gross--Kudla formula \cite[Prop. 10.8]{grku}, which gives an explicit expression for the central value $L(F,2)$.
Under Assumption \ref{sign}, the Beilinson--Bloch conjecture \eqref{bb} predicts that the algebraic rank of the $F$-isotypic component of $\CH^2(X_0(p)^3)_0(\Q)$ is even. Comparing with the situation of Heegner points on modular curves studied in \cite{GZ}, it seems reasonable to expect that the $F$-isotypic component of $\Delta_{\gks}(e)$ is torsion, for all $e\in X(\Q)$. While this appears to be difficult to show directly in the Chow group, we can prove the corresponding statement for the image of the cycle under the complex Abel--Jacobi map 
\begin{equation}\label{torsiontripleCAJ}
\AJ_{X^3} : \CH^2(X^3)_0(\C)\lra J^2(X^3/\C) :=\frac{\Fil^{2} H^{3}_{\dR}(X^3/\C)^\vee}{\im H_{3}(X^3(\C), \Z)},
\end{equation}
whose target is the second Griffiths intermediate Jacobian of $X^3$.
This map is a higher dimensional generalisation of the familiar Abel--Jacobi isomorphism for curves. It is defined by the integration formula 
\[
\AJ_{X^3}(Z)(\alpha):=\int_{\partial^{-1}(Z)} \alpha, \qquad \text{for all } \alpha\in \Fil^{2} H^{3}_{\dR}(X^3/\C),
\]
where $\partial^{-1}(Z)$ denotes any continuous $3$-chain in $X^3(\C)$ whose image under the boundary map $\partial$ is $Z$.
The aim of this section is to prove the following main result.
\begin{theorem}\label{GKStorsion}
Let $f_1, f_2, f_3$ be three normalised eigenforms in $S_2(\Gamma_0(p))$, let $F=f_1\otimes f_2\otimes f_3$ be their triple product, and suppose that $F$ satisfies Assumption \ref{sign}.
Then $\AJ_{X^3}((t_F)_*(\Delta_{\gks}(e)))$ is torsion in $J^2(X^3/\C)$, for all $e\in X(\Q)$.
\end{theorem}

\begin{remark}\label{rem:GKStorsion}
Similar arguments to the ones presented in the proof of Theorem \ref{GKStorsion} below can be used to prove that the image of $(t_F)_*(\Delta_{\gks}(e))$ under Bloch's \cite{bloch} $\ell$-adic \'etale Abel--Jacobi map 
\begin{equation}\label{tripleet}
\AJ^{\et}_{X^3} : \CH^2(X^3)_0(\Q) \lra H^1(\Q, H^{3}_{\et}(X^3_{\bar{\Q}}, \Q_\ell(2)))
\end{equation}
is torsion, when the global root number is $W(F)=+1$. 
It is conjectured that for any smooth proper variety over a number field, and for any prime $\ell$, the $\ell$-adic Abel--Jacobi maps in any codimension are injective up to torsion \cite[Conj. 9.15]{jannsen}. Thus, conditional on this conjecture, $(t_F)_*(\Delta_{\gks}(e))$ is torsion in the Chow group.
\end{remark}

The rest of this section constitutes the proof of Theorem \ref{GKStorsion}.
We distinguish different situations depending on the genus $g_X$ of $X$, which we recall is given by the formula (\ref{g0}).
The curve $X$ has genus zero exactly when $p\in \{ 2, 3, 5, 7, 13\}$. In this case, the space of cusp forms $S_2(\Gamma_0(p))$ is trivial, and there is no triple product $L$-function to consider in the first place. We have $\Delta_{\gks}(e)=0$ in $\CH^2(X^3)_0(\Q)$, as the cycle class map is injective in this case \cite[Prop. 4.1]{grsc}.

\subsection{The genus one case}

Suppose that $g_X=1$, i.e., $p\in \{ 11, 17, 19\}$. 
In this case, $X$ is an elliptic curve over $\Q$ of Mordell--Weil rank $0$. 
For all $e\in X(\Q)$, we have $6\Delta_{\gks}(e)=0$ in $\CH^2(X^3)_0(\Q)$ \cite[Cor. 4.7]{grsc}.
On the $L$-function side, $f_1=f_2=f_3=f$ is the normalised eigenform corresponding to the elliptic curve $X$. By \cite[$(11.8)$]{grku} the triple product $L$-function decomposes as 
\[
L(F, s)=L(\sym^3 f, s)L(f, s-1)^2.
\]
Note that $W(F)=a_p(f)^3=a_p(f)=W(f)=+1$ by \eqref{rem:tripleroot} and the fact that the sign of the functional equation of $L(f, s)$ centred at $s=1$ is equal to $+1$, since $X$ has Mordell--Weil rank $0$. 
For each $p\in \{ 11, 17, 19\}$, we have $L(F, 2)\neq 0$ \cite[Tables 12.5, 12.6, \& 12.7]{grku}. In other words, $\ord_{s=2}(L(F, s))=0$. The fact that $\Delta_{\gks}(e)$ is torsion in the Chow group is therefore consistent with the Beilinson--Bloch conjecture \eqref{bb}. 

\subsection{The higher genus case}

Suppose that $g_X\geq 2$. It will be convenient to sometimes view the Atkin--Lehner involution $w_p$ of Section \ref{s:hecke} as a correspondence by taking its graph. By slight abuse of notation, we will write $w_p \in \corr^0(X, X)$. The operator $w_p$ naturally belongs to the Hecke algebra $\T$, by \eqref{hecke} and \eqref{endo}, and commutes with the Hecke operators. 
The modular forms $f_j$, with $j\in \{1, 2, 3\}$, are eigenforms for the operator $w_p$ with eigenvalues given by $-a_p(f_j)$ respectively (see Section \ref{s:hecke}). 

Consider the involution $u_p:=w_p\times w_p\times w_p$ of $X^3$. By taking its graph, it may be viewed as a correspondence, and we write again $u_p\in \corr^0(X^3, X^3)$, by slight abuse of notation. Note that, as correspondences, we have
 \[
 u_p = w_p\otimes w_p \otimes w_p :=\pr_{14}^*(w_p)\cdot \pr_{25}^*(w_p)\cdot \pr_{36}^*(w_p) \in \corr^{0}(X^3, X^3).
 \]
 The map $u_p$ induces an involution on cohomology via pull-back, hence an involution on the space of cusp forms of weight $(2, 2, 2)$ for $\Gamma_0(p)^3$. By \eqref{rem:tripleroot}, we see that 
\begin{equation}\label{eq:Fup}
u_p^*(F)=-W(F)\cdot F.
\end{equation}

\begin{lemma}\label{lem:up}
We have
$
(u_p)_*(\Delta_{\gks}(e))=\Delta_{\gks}(w_p(e)),
$
for any point $e$ of $X$.
\end{lemma}

\begin{proof}
Note that the induced map $(u_p)_* : \CH^2(X^3) \lra \CH^2(X^3)$ on Chow groups simply maps a cycle to its image under $u_p$.
We have $u_p(\Delta)=\Delta$, since $u_p$ is an automorphism of $X^3$. However, for any proper subset $T\subset \{ 1,2,3\}$, we have $u_p(P_T(e)_*(\Delta))=P_T(w_p(e))_*(\Delta)$. 
\end{proof}

\begin{proposition}\label{prop:AJup}
Let $f_1, f_2, f_3$ be three normalised eigenforms in $S_2(\Gamma_0(p))$, let $F=f_1\otimes f_2\otimes f_3$ be their triple product, and suppose that $F$ satisfies Assumption \ref{sign}. For any point $e$ of $X$, we have 
$$
\AJ_{X^3}((t_F)_*(\Delta_{\gks}(e)))=-\AJ_{X^3}((t_F)_*(\Delta_{\gks}(w_p(e)))).
$$
\end{proposition}

\begin{proof}
By functoriality of Abel--Jacobi maps with respect to correspondences, we have 
\begin{equation}\label{AJ1}
\AJ_{X^3}((u_p)_*(t_F)_*(\Delta_{\gks}(e)))=(u_p^*)^\vee \AJ_{X^3}((t_F)_*(\Delta_{\gks}(e))).
\end{equation}
By \eqref{hecke} and \eqref{endo}, $w_p$ commutes with $t_{f_j}$, for $j\in \{1, 2, 3\}$, as correspondences, implying that
\begin{equation*}
t_F\circ u_p = (t_{f_1}\circ w_p) \otimes (t_{f_2}\circ w_p) \otimes (t_{f_3}\circ w_p) 
= (w_p \circ t_{f_1}) \otimes (w_p \circ t_{f_2}) \otimes (w_p \circ t_{f_3}) 
= u_p \circ t_F,
\end{equation*}
as elements in $\corr^0(X^3, X^3)$. In particular, 
using Lemma \ref{lem:up}, we obtain
\[
(u_p)_*(t_F)_*(\Delta_{\gks}(e))=(t_F)_*(u_p)_*(\Delta_{\gks}(e))=(t_F)_*(\Delta_{\gks}(w_p(e))).
\]
The left hand side of \eqref{AJ1} is thus equal to $\AJ_{X^3}((t_F)_*(\Delta_{\gks}(w_p(e))))$.

On the other hand, $\AJ_{X^3}((t_F)_*(\Delta_{\gks}(e)))$ lies in $(t_F^*)^\vee(J^2(X^3/\C))$ by functoriality of the complex Abel--Jacobi map with respect to correspondences, that is, in the $F$-isotypic Hecke component of the intermediate Jacobian. The triple product Hecke algebra $\T^{\otimes 3}$ acts via correspondences on the latter by multiplication by the Hecke eigenvalues of $F$. 
More precisely, for any $\alpha\in \Fil^{2} H^{3}_{\dR}(X^3/\C)$, we have the equality
\[
(u_p^*)^\vee \AJ_{X^3}((t_F)_*(\Delta_{\gks}(e)))(\alpha)=\AJ_{X^3}(\Delta_{\gks}(e))(u_p^*(t_F^*(\alpha))).
\]
The operator $u_p\in \T^{\otimes 3}$ acts via pull-back on the $F$-isotypic component $(t_F)^*H_{\dR}^3(X^3/\C)$ as multiplication by $-W(F)$ by \eqref{eq:Fup}. In particular, we have $u_p^*(t_F^*(\alpha))=-W(F)t_F^*(\alpha)$.
By Assumption \ref{sign}, the right hand side of \eqref{AJ1} is therefore given by
\[
(u_p^*)^\vee \AJ_{X_0(p)^3}((t_F)_*(\Delta_{\gks}(e)))=- \AJ_{X_0(p)^3}((t_F)_*(\Delta_{\gks}(e))).
\]
The result follows.
\end{proof}

Mazur \cite[Thm. 1]{mazur78} proved, for $g_X\geq 2$ and $p\not\in \{ 37, 43, 67, 163\}$, that $X(\Q)=\{ \xi_\infty, \xi_0 \},$ where we recall that $\xi_\infty$ and $\xi_0$ denote the two cusps of $X$. Moreover, the modular curve $X_0(37)$ has two non-cuspidal $\Q$-rational points, while $X_0(p)$ has a unique non-cuspidal $\Q$-rational point, for $p\in\{ 43, 67, 163\}$. 

\begin{corollary}
Let $f_1, f_2, f_3$ be three normalised eigenforms in $S_2(\Gamma_0(p))$, let $F=f_1\otimes f_2\otimes f_3$ be their triple product, and suppose that $F$ satisfies Assumption \ref{sign}.
If $p$ belongs to $\{ 43, 67, 163 \}$, and $e$ denotes the unique non-cuspidal $\Q$-rational point of $X$, then $2 \AJ_{X^3}((t_F)_*(\Delta_{\gks}(e)))=0.$
\end{corollary}

\begin{proof}
The involution $w_p$ maps $\Q$-rational points to $\Q$-rational points and permutes the two cusps $\xi_\infty$ to $\xi_0$. It therefore fixes the non-cuspidal point $e$, and the result follows from Proposition \ref{prop:AJup}. \end{proof}

\begin{corollary}\label{corocusp}
Let $f_1, f_2, f_3$ be three normalised eigenforms in $S_2(\Gamma_0(p))$, let $F=f_1\otimes f_2\otimes f_3$ be their triple product, and suppose that $F$ satisfies Assumption \ref{sign}.
If $g_X\geq 2$ and $n$ denotes the numerator of $(p-1)/12$, then
\[
2n \AJ_{X^3}((t_F)_*(\Delta_{\gks}(\xi_\infty)))=2n \AJ_{X^3}((t_F)_*(\Delta_{\gks}(\xi_0)))=0. 
\] 
\end{corollary}

\begin{proof}
Gross and Schoen \cite[Prop. 3.6]{grsc} have constructed a correspondence $\Delta_{\gks}$ in $\corr^{1}(X, X^3)$ with the property that the induced push-forward map 
\begin{equation}\label{grsccor}
\Delta_{\gks, *} : \CH^1(X)=\Pic(X) \lra \CH^2(X^3) 
\end{equation}
sends the rational equivalence class of a divisor $\sum m(e)e$ to $\sum m(e)\Delta_{\gks}(e)$. In particular,
the cycle $\Delta_{\gks}(\xi_\infty)-\Delta_{\gks}(\xi_0)$ in $\CH^2(X^3)_0(\Q)$ depends only on the class of the divisor $(\xi_\infty)-(\xi_0)$ in $\CH^1(X)_0(\Q)=J(\Q)$. By \cite[Thm. 1]{mazur77}, the degree zero divisor $(\xi_\infty)-(\xi_0)$ is torsion of order $n$ in the Jacobian $J$. It follows that
$$
n(\Delta_{\gks}(\xi_\infty)-\Delta_{\gks}(\xi_0))=\Delta_{\gks, *}(n((\xi_\infty)-(\xi_0)))=0 
$$
in $\CH^2(X^3)_0(\Q),$
and in particular
\[
n(\AJ_{X^3}((t_F)_*(\Delta_{\gks}(\xi_\infty)))-\AJ_{X^3}((t_F)_*(\Delta_{\gks}(\xi_0))))=0
\]
in $J^{2}(X^3/\C)$.
The involution $w_p$ permutes the cusps $\xi_\infty$ and $\xi_0$. By Proposition \ref{prop:AJup}, we thus have
$
\AJ_{X^3}((t_F)_*(\Delta_{\gks}(\xi_\infty)))=-\AJ_{X^3}((t_F)_*(\Delta_{\gks}(\xi_0))),
$
which concludes the proof.
\end{proof}

\subsection{The case $p=37$}

To complete the proof of Theorem \ref{GKStorsion}, the only remaining case is the one where $p=37$ and the chosen base point is a non-cuspidal $\Q$-rational point. The curve $X_0(37)$ has been extensively studied by Mazur and Swinnerton-Dyer \cite[\S 5]{mazurSD}. It has genus $2$ and is therefore hyperelliptic. Its hyperelliptic involution will be denoted by $S$. In particular, for all points $e$ of $X_0(37)$, we have $6\Delta_{\gks}(e)=0$ in the Griffiths group $\Gr^2(X_0(37)^3)$ of null-homologous algebraic cycles modulo algebraic equivalence \cite[Cor. 4.9]{grsc}. The involution $S$ is distinct from the Atkin--Lehner involution $w_{37}$, as the quotient $X_0(37)/w_{37}$ has genus $1$. Since $S$ commutes with every automorphism of $X_0(37)$ \cite[p. 27]{mazurSD}, it commutes in particular with $w_{37}$, and we can define another involution $T=S\circ w_{37}=w_{37}\circ S$. Let $\gamma_0=T(\xi_0)$ and $\gamma_\infty=T(\xi_\infty)$ be the images of the two cusps by $T$. By \cite[Prop. 2]{mazurSD}, we have 
\begin{equation}\label{xo37}
X_0(37)(\Q)=\{ \xi_0, \xi_\infty, \gamma_0, \gamma_\infty \} \qquad \text{ and } \qquad w_{37}(\gamma_0)=\gamma_\infty.
\end{equation}

The involution $S$ has $6$ fixed points, none of which are rational over $\Q$. By \cite[Prop. 4.8]{grsc}, $6\Delta_{\gks}(e)=0$ in $\CH^2(X_0(37)^3)_0$ if $e$ is a fixed point of $S$. 
By \cite[p. 29]{mazurSD}, the two fixed points $\alpha_1$ and $\alpha_2$ of $w_{37}$ are Galois conjugates defined over $\Q(\sqrt{37})$. We have the following result.
\begin{corollary}
Let $f_1, f_2, f_3$ be three normalised eigenforms in $S_2(\Gamma_0(37))$, let $F=f_1\otimes f_2\otimes f_3$ be their triple product, and suppose that $F$ satisfies Assumption \ref{sign}.
The images under the complex Abel--Jacobi map $\AJ_{X_0(37)^3}$ of the cycles $(t_F)_*\Delta_{\gks}(\alpha_1)$ and $(t_F)_*\Delta_{\gks}(\alpha_2)$ are $2$-torsion in the intermediate Jacobian $J^2(X_0(37)^3/\C)$.
\end{corollary} 

\begin{proof}
This is an immediate consequence of Proposition \ref{prop:AJup}, given that $\alpha_1$ and $\alpha_2$ are the fixed points of $w_{37}$.
\end{proof}

We complete the proof of Theorem \ref{GKStorsion}. 
\begin{corollary}
Let $f_1, f_2, f_3$ be three normalised eigenforms in $S_2(\Gamma_0(37))$, let $F=f_1\otimes f_2\otimes f_3$ be their triple product, and suppose that $F$ satisfies Assumption \ref{sign}. Then $$6 \AJ_{X_0(37)^3}((t_F)_*(\Delta_{\gks}(\gamma_0)))=6\AJ_{X_0(37)^3}((t_F)_*(\Delta_{\gks}(\gamma_\infty)))=0.$$
\end{corollary}

\begin{proof}
By \eqref{xo37}, the Atkin--Lehner involution $w_{37}$ interchanges $\gamma_0$ and $\gamma_\infty$. By Proposition \ref{prop:AJup}, we have $\AJ_{X_0(37)^3}((t_F)_*(\Delta_{\gks}(\gamma_0)))=-\AJ_{X_0(37)^3}((t_F)_*(\Delta_{\gks}(\gamma_\infty)))$. The element 
\[
2\AJ_{X_0(37)^3}((t_F)_*(\Delta_{\gks}(\gamma_0)))=\AJ_{X_0(37)^3}((t_F)_*(\Delta_{\gks}(\gamma_0)-\Delta_{\gks}(\gamma_\infty)))
\]
in $J^2(X_0(37)^3/\C)$
depends only on the class of $(\gamma_0)-(\gamma_\infty)$ in $J_0(37)(\Q)$ by \eqref{grsccor}. But this class is the image of the class of $(\xi_0)-(\xi_\infty)$ by the involution of $J_0(37)$ obtained from $T$ by push-forward. The latter class has order equal to the numerator of $(37-1)/12=3$ \cite[Thm. 1]{mazur77}. 
\end{proof}

\section{Chow--Heegner points}\label{s:CH}

Let $f$ be a normalised eigenform in $S_2(\Gamma_0(p))$ with rational coefficients, and let $E_f$ be the elliptic curve associated with $f$ by the Eichler--Shimura construction \cite{shimura58}. In particular, there is a quotient map
$
\pi_f : J \lra E_f, 
$
induced by the idempotent correspondence $t_f$ in $\corr^0(X, X)_\Q$ of Section \ref{s:motivef}. In this special case of rational coefficients, note that we have the equality of motives $E_f=M(f)=M([f])=(X, t_f, 0)$.

\begin{remark}
To the best of the author's knowledge, it is unknown whether there are finitely or infinitely many elliptic curves over $\Q$ of prime conductor. It is a result of Setzer \cite[Thm. 2]{setzer} that, given a prime $p$ distinct from $2, 3$, and $17$, there is an elliptic curve of conductor $p$ over $\Q$ with a rational $2$-torsion point if and only if $p=u^2+64$ for some rational integer $u$. A conjecture of Hardy and Littlewood \cite[Conj. F]{hardylittlewood} implies that there are infinitely many values of $u$ such that $u^2+64$ is prime. Thus, conditional on this conjecture of Hardy and Littlewood, there are infinitely many primes $p$ which occur as the conductor of an elliptic curve over $\Q$. This is explained in detail in the preprint \cite{howejoshi}.
\end{remark}

Let $g$ be an auxiliary normalised eigenform in $S_2(\Gamma_0(p))$, distinct from $f$. Recall from Section \ref{s:motivef} the correspondence $t_{[g]}\in \corr^0(X, X)_\Q$, which cuts out the motive $M([g])=(X, t_{[g]}, 0)=A_{[g]}$. 
Consider the correspondence 
$$
\Pi_{[g]}:= \pr_{12}^*(t_{[g]}) \cdot \pr_{34}^*(\Delta)\in \CH^2(X^4)(\Q)\otimes \Q,
$$
where $\Delta\in \CH^1(X^2)(\Q)$ is the diagonal cycle. After clearing denominators, we may and will consider $\Pi_{[g]}$ as an element of $\corr^{-1}(X^3, X)$, which in turn 
induces a push-forward map
\[
\Pi_{[g], *} : \CH^2(X^3)_0(L) \lra \CH^1(X)_0(L)=J(L),
\]
for any field extension $L$ of $\Q$.
By composing correspondences, we can define 
\begin{equation}\label{pigf}
\Pi_{[g], f}:=t_f \circ \Pi_{[g]} = \pr_{12}^*(t_{[g]}) \cdot \pr_{34}^*(t_f) \in \corr^{-1}(X^3, E_f).
\end{equation}
This induces, in the terminology of \cite{bdp2}, a generalised modular parametrisation
\[
\Pi_{[g],f, *} = \pi_f \circ \Pi_{[g], *} : \CH^2(X^3)_0(L) \lra E_f(L). 
\]

\begin{remark}
Instead of defining the correspondence $\Pi_{[g]}$ as $\pr_{12}^*(t_{[g]})\cdot \pr_{34}^*(\Delta)$, one could imagine proposing to use $\pr_{12}^*(t_{[g]})\cdot \pr_{34}^*(t_{[g]})$. One checks however that 
$$\pr_{12}^*(t_{[g]})\cdot \pr_{34}^*(t_{[g]})= t_{[g]}\circ (\pr_{12}^*(t_{[g]})\cdot \pr_{34}^*(\Delta)),$$ 
hence $t_f\circ (\pr_{12}^*(t_{[g]})\cdot \pr_{34}^*(t_{[g]}))= (t_f\circ t_{[g]})\circ (\pr_{12}^*(t_{[g]})\cdot \pr_{34}^*(\Delta))$. Since $f$ and $g$ are not $\Gal(\bar{\Q}/\Q)$ conjugates, we have $\pi_f \circ \pi_{[g]}=0$ in $\End^0_\Q(J)$. In particular, the generalised modular parametrisation $(t_f \circ (\pr_{12}^*(t_{[g]})\cdot \pr_{34}^*(t_{[g]})))_* : \CH^2(X^3)_0\lra E_f$ is the zero map in this case. 
\end{remark}

Using the modular parametrisation $\Pi_{[g],f, *}$, the cycle $\Delta_{\gks}(\xi_\infty)\in \CH^2(X^3)_0(\Q)$, and the conjectural formula \eqref{conj:grku}, Darmon, Rotger, and Sols proved the following concerning the Chow--Heegner point 
\begin{equation}\label{CHP}
P_{g,f}(\xi_\infty):=\Pi_{[g],f,*}(\Delta_{\gks}(\xi_\infty))=\pi_f(\Pi_{[g],*}(\Delta_{\gks}(\xi_\infty)))\in E_f(\Q),
\end{equation}
building on the work of Yuan, Zhang, and Zhang \cite{yzz}.
\begin{theorem}\label{thm:DRS}
Assume that $W(f)=-1$ and $W(\sym^2 g\otimes f)=+1$. Then the Chow--Heegner point $P_{g, f}(\xi_\infty)$ has infinite order in $E_f(\Q)$ if and only if 
\[
\ord_{s=1} L(f, s)=1 \qquad \text{ and } \qquad \ord_{s=2} L(\sym^2 (g^\sigma)\otimes f, s)=0, \qquad \forall \:\sigma : K_g \hookrightarrow \C.
\]
\end{theorem}

\begin{proof}
This is \cite[Thm. 3.7]{DRS}.
\end{proof}

\begin{remark}
Note that the triple product $L$-function attached to $(g,g,f)$ decomposes as 
\[
L(g,g,f,s)=L(f, s-1)L(\sym^2 g\otimes f, s),
\]
and therefore the assumptions of Theorem \ref{thm:DRS} imply in particular that $W(g,g,f)=-1$.
\end{remark}

In the complementary setting where $W(g,g,f)=+1$, we prove the following.

\begin{theorem}\label{CHPtors}
If $E_f$ admits split multiplicative reduction at $p$, then the Chow--Heegner point $P_{g, f}(e):=\Pi_{[g], f, *}(\Delta_{\gks}(e))$ is torsion in $E_f(\Q)$, for all $e\in X(\Q)$.
\end{theorem}

\begin{proof}
Recall from Section \ref{s:motivef} that $t_{[g]}=\sum_{h\in [g]} t_h$, and thus 
\[ 
t_{[g]}\otimes t_{[g]} \otimes t_{f}=\sum_{h_1, h_2\in [g]} t_{h_1}\otimes t_{h_2} \otimes t_{f}.
\]
By \eqref{rem:tripleroot}, for any $h_1, h_2\in [g]$, the global root number of the triple product $L$-function $L(h_1,h_2,f, s)$ is given by
$W(h_1, h_2, f)=a_p(h_1)a_p(h_2)a_p(f)$.
The $p$-th Fourier coefficient of a normalised cuspidal eigenform is the negative of the $w_p$-eigenvalue of the form, hence it belongs to $\{ \pm 1 \}$. In particular, since this coefficient belongs to $\Q$, it is fixed by the action of $\Gal(\bar{\Q}/\Q)$, and thus $a_p(g)=a_p(h_1)=a_p(h_2)\in \{ \pm 1 \}$. Consequently, it follows that $W(h_1, h_2, f)=a_p(f)=a_p(E_f)$. We have $a_p(E_f)=1$, since $E_f$ admits split multiplicative reduction at $p$, and thus the triple $(h_1,h_2, f)$ satisfies Assumption \ref{sign}. By Theorem \ref{GKStorsion}, for any $e\in X(\Q)$, $\AJ_{X^3}((t_{h_1}\otimes t_{h_2} \otimes t_{f})_*(\Delta_{\gks}(e)))$ is torsion in the intermediate Jacobian $J^2(X^3/\C)$. It follows that
$\AJ_{X^3}((t_{[g]} \otimes t_{[g]} \otimes t_{f})_*(\Delta_{\gks}(e)))$ is torsion in $J^2(X^3/\C)$.

Define the cycle
\[
\Pi:=\pr_{12}^*(\Delta)\cdot \pr_{34}^*(\Delta)\in \CH^2(X_0(p)^4).
\]
Viewing $t_{[g]} \otimes t_{[g]} \otimes t_{f}$ as an element of $\corr^0(X^3, X^3)_\Q$, and $\Pi$ as an element of $\corr^{-1}(X^3, X)$, we compute their composition to obtain 
\begin{equation}\label{eq:corr}
\Pi \circ (t_{[g]} \otimes t_{[g]} \otimes t_{f})
= \pr_{12}^*(t_{[g]} \circ t_{[g]})\cdot \pr_{34}^*(t_f) = \pr_{12}^*(t_{[g]})\cdot \pr_{34}^*(t_f)=\Pi_{[g], f},
\end{equation}
as elements of $\corr^{-1}(X^3, X)_\Q$. Note the use of the fact that $t_{[g]}$ is an idempotent correspondence, i.e., $t_{[g]}\circ t_{[g]}=t_{[g]}$. A similar calculation is carried out in \cite[\S 3]{DRS}.
We deduce the equality of points in $E_f(\Q)$ 
\begin{equation}\label{pointequal}
\Pi_*(t_{[g]} \otimes t_{[g]} \otimes t_{f})_*(\Delta_{\gks}(e))=P_{g, f}(e).
\end{equation}
By the functoriality of Abel--Jacobi maps with respect to correspondences, we have a commutative diagram
  \begin{equation*}
\begin{tikzcd}
 \CH^2(X^3)_0(\C) \arrow{r}{\AJ_{X^3}} \arrow{d}[swap]{\Pi_{[g], f,*}} 
 & J^2(X^3/\C) \arrow{d}{(\Pi_{[g], f}^*)^\vee} \\
 E_f(\C) \arrow{r}{\sim}[swap]{\AJ_{E_f}} 
 & J^1(E_f/\C).
\end{tikzcd}
\end{equation*}
Here $J^1(E_f/\C)=H^0(E_f(\C), \Omega^1_{E_f})^\vee \slash \im H_1(E_f(\C), \Z)$ is the Jacobian of $E_f$, and $\AJ_{E_f}$ is the classical Abel--Jacobi isomorphism for the elliptic curve $E_f$, given by the integration formula
\[
\AJ_{E_f}(P)(\alpha):= \int_{\cO}^P \alpha, \qquad \text{ for } \alpha\in  H^0(E_f(\C), \Omega^1_{E_f}),
\]
where $\cO$ is the origin of $E_f$.
 In particular, we have the equalities
\begin{align}
\AJ_{E_f}(P_{g, f}(e)) 
& =(\Pi_{[g], f}^*)^\vee(\AJ_{X^3}(\Delta_{\gks}(e)))\nonumber \\
& = (\Pi^*)^\vee ((t_{[g]} \otimes t_{[g]} \otimes t_{f})^*)^\vee \AJ_{X^3}(\Delta_{\gks}(e)) \label{eq} \\
& = (\Pi^*)^\vee \AJ_{X^3}((t_{[g]} \otimes t_{[g]} \otimes t_{f})_*(\Delta_{\gks}(e))). \nonumber
\end{align}
In the second equality we used \eqref{eq:corr}, and in the third equality we used the functoriality of $\AJ_{X^3}$ with respect to the correspondence $t_{[g]} \otimes t_{[g]} \otimes t_{f}$.
Since $\AJ_{X^3}((t_{[g]} \otimes t_{[g]} \otimes t_{f})_*(\Delta_{\gks}(e)))$ is torsion, the result follows from the fact that $\AJ_{E_f}$ is an isomorphism. 
\end{proof}

\begin{remark}\label{rem:daub}
Theorem \ref{CHPtors} with $e=\xi_\infty$ is a special case of \cite[Thm. 3.3.8]{daubthesis}. In his thesis \cite{daubthesis}, Daub proved more generally for composite level $N$ that if the local root number $W_p(g,g,f)=-1$ for some $p \: \vert \: N$, then the resulting Chow--Heegner points based at $\xi_\infty$ are torsion. His proof relies on an identification of these points with Zhang points \cite{zhang10}. As explained in the introduction (Section \ref{s:squarefree}), our method works for composite level $N$ and fully recovers \cite[Thm. 3.3.8]{daubthesis}.
\end{remark}

\begin{remark}\label{rem:stein}
Techniques have been developed in \cite{ddlr} to numerically calculate Chow--Heegner points associated with modified diagonal cycles on triple products of modular curves. The algorithms are based on a formula for the image of these cycles under the complex Abel--Jacobi map \eqref{torsiontripleCAJ} proved in \cite{DRS}. Most of the examples calculated in \cite{ddlr} concern the situation where the elliptic curve $E_f$ has algebraic rank equal to $1$. In particular, the global root number of $E_f$ is $-1$, and this is not the setting studied in the present paper. However, in the appendix of \cite{ddlr} by Stein, some examples are computed for which the rank of $E_f$ is $0$. In particular, it is verified numerically that $P_{g,f}(\xi_\infty)$ is a point of order $3$ in $E_f(\Q)$, where $E_f$ has Cremona label $37$b, and $A_{[g]}$ is the elliptic curve with Cremona label $37$a. Note that Corollary \ref{corocusp}, together with \eqref{eq}, implies that the order of $P_{g,f}(\xi_\infty)$ divides $6$, since $(37-1)/12=3$. As a consequence of the non-triviality of $P_{g,f}(\xi_\infty)$, the Abel--Jacobi image $\AJ_{X_0(37)^3}((t_{[g]} \otimes t_{[g]} \otimes t_{f})_*(\Delta_{\gks}(\xi_\infty)))$ is non-trivial by \eqref{eq}, and in particular the cycle $(t_{[g]} \otimes t_{[g]} \otimes t_{f})_*(\Delta_{\gks}(\xi_\infty)$ is non-trivial. By Corollary \ref{corocusp}, $\AJ_{X_0(37)^3}((t_{[g]} \otimes t_{[g]} \otimes t_{f})_*(\Delta_{\gks}(\xi_\infty)))$ is a non-trivial torsion element in the intermediate Jacobian. This proves Corollary \ref{coromain}. It would be interesting do perform more calculations of Chow--Heegner points that fall under the assumptions of the present paper.
\end{remark}

\subsection*{Acknowledgements}
It is a pleasure to thank Henri Darmon, Benedict Gross, and Ari Shnidman for helpful comments and corrections, as well as Congling Qiu and Wei Zhang for answering questions related to their work.  
The author was partially supported by the Institut des Sciences Math\'ematiques (ISM) au Qu\'ebec while at McGill University, and by an Emily Erskine Endowment Fund Postdoctoral Researcher Fellowship while at the Hebrew University of Jerusalem.

\phantomsection
\bibliographystyle{plain}
\bibliography{GKStorsion.bib}

\begin{thebibliography}{10}

\bibitem{atkinlehner}
A.~O.~L. Atkin and J.~Lehner.
\newblock Hecke operators on {$\Gamma _{0}(m)$}.
\newblock {\em Math. Ann.}, 185:134--160, 1970.

\bibitem{bdp2}
M.~Bertolini, H.~Darmon, and K.~Prasanna.
\newblock Chow-{H}eegner points on {CM} elliptic curves and values of
  {$p$}-adic {$L$}-functions.
\newblock {\em Int. Math. Res. Not. IMRN}, (3):745--793, 2014.

\bibitem{bloch}
S.~Bloch.
\newblock Algebraic cycles and values of {$L$}-functions.
\newblock {\em J. Reine Angew. Math.}, 350:94--108, 1984.

\bibitem{ddlr}
H.~Darmon, M.~Daub, S.~Lichtenstein, and V.~Rotger.
\newblock Algorithms for {C}how-{H}eegner points via iterated integrals.
\newblock {\em Math. Comp.}, 84(295):2505--2547, 2015.

\bibitem{ddt}
H.~Darmon, F.~Diamond, and R.~Taylor.
\newblock Fermat's last theorem.
\newblock In {\em Current developments in mathematics, 1995 ({C}ambridge,
  {MA})}, pages 1--154. Int. Press, Cambridge, MA, 1994.

\bibitem{DRS}
H.~Darmon, V.~Rotger, and I.~Sols.
\newblock Iterated integrals, diagonal cycles and rational points on elliptic
  curves.
\newblock In {\em Publications math\'{e}matiques de {B}esan\c{c}on. {A}lg\`ebre
  et th\'{e}orie des nombres, 2012/2}, volume 2012/ of {\em Publ. Math.
  Besan\c{c}on Alg\`ebre Th\'{e}orie Nr.}, pages 19--46. Presses Univ.
  Franche-Comt\'{e}, Besan\c{c}on, 2012.

\bibitem{daubthesis}
M.~W. Daub.
\newblock {\em Complex and p-adic {C}omputations of {C}how-{H}eegner {P}oints}.
\newblock ProQuest LLC, Ann Arbor, MI, 2013.
\newblock Thesis (Ph.D.)--University of California, Berkeley.

\bibitem{del73}
P.~Deligne.
\newblock Les constantes des \'{e}quations fonctionnelles des fonctions {$L$}.
\newblock In {\em Modular functions of one variable, {II} ({P}roc. {I}nternat.
  {S}ummer {S}chool, {U}niv. {A}ntwerp, {A}ntwerp, 1972)}, pages 501--597.
  Lecture Notes in Math., Vol. 349, 1973.

\bibitem{deligne}
P.~Deligne.
\newblock Valeurs de fonctions {$L$} et p\'{e}riodes d'int\'{e}grales.
\newblock In {\em Automorphic forms, representations and {$L$}-functions
  ({P}roc. {S}ympos. {P}ure {M}ath., {O}regon {S}tate {U}niv., {C}orvallis,
  {O}re., 1977), {P}art 2}, Proc. Sympos. Pure Math., XXXIII, pages 313--346.
  Amer. Math. Soc., Providence, R.I., 1979.
\newblock With an appendix by N. Koblitz and A. Ogus.

\bibitem{grku}
B.~H. Gross and S.~S. Kudla.
\newblock Heights and the central critical values of triple product
  {$L$}-functions.
\newblock {\em Compositio Math.}, 81(2):143--209, 1992.

\bibitem{grsc}
B.~H. Gross and C.~Schoen.
\newblock The modified diagonal cycle on the triple product of a pointed curve.
\newblock {\em Ann. Inst. Fourier (Grenoble)}, 45(3):649--679, 1995.

\bibitem{GZ}
B.~H. Gross and D.~B. Zagier.
\newblock Heegner points and derivatives of {$L$}-series.
\newblock {\em Invent. Math.}, 84(2):225--320, 1986.

\bibitem{hardylittlewood}
G.~H. Hardy and J.~E. Littlewood.
\newblock Some problems of `{P}artitio numerorum'; {III}: {O}n the expression
  of a number as a sum of primes.
\newblock {\em Acta Math.}, 44(1):1--70, 1923.

\bibitem{howejoshi}
S.~Howe and K.~Joshi.
\newblock Asymptotics of conductors of elliptic curves over {$\Q$}.
\newblock Preprint, arXiv:1201.4566, 2015.

\bibitem{jannsen}
U.~Jannsen.
\newblock {\em Mixed motives and algebraic {$K$}-theory}, volume 1400 of {\em
  Lecture Notes in Mathematics}.
\newblock Springer-Verlag, Berlin, 1990.
\newblock With appendices by S. Bloch and C. Schoen.

\bibitem{kani}
E.~Kani.
\newblock Endomorphisms of {J}acobians of modular curves.
\newblock {\em Arch. Math. (Basel)}, 91(3):226--237, 2008.

\bibitem{kenku}
M.~A. Kenku.
\newblock On the modular curves {$X_{0}(125)$}, {$X_{1}(25)$} and
  {$X_{1}(49)$}.
\newblock {\em J. London Math. Soc. (2)}, 23(3):415--427, 1981.

\bibitem{langebirkenhake}
H.~Lange and C.~Birkenhake.
\newblock {\em Complex abelian varieties}, volume 302 of {\em Grundlehren der
  Mathematischen Wissenschaften [Fundamental Principles of Mathematical
  Sciences]}.
\newblock Springer-Verlag, Berlin, 1992.

\bibitem{mythesis}
D.~T.-B.~G. Lilienfeldt.
\newblock {\em {A}lgebraic cycles and {D}iophantine geometry: generalised
  {H}eegner cycles, quadratic {C}habauty and diagonal cycles}.
\newblock 2021.
\newblock Thesis (Ph.D.)--McGill University.

\bibitem{manin}
J.~I. Manin.
\newblock Parabolic points and zeta functions of modular curves.
\newblock {\em Izv. Akad. Nauk SSSR Ser. Mat.}, 36:19--66, 1972.

\bibitem{mazur77}
B.~Mazur.
\newblock Modular curves and the {E}isenstein ideal.
\newblock {\em Inst. Hautes \'{E}tudes Sci. Publ. Math.}, (47):33--186 (1978),
  1977.
\newblock With an appendix by Mazur and M. Rapoport.

\bibitem{mazur78}
B.~Mazur.
\newblock Rational isogenies of prime degree (with an appendix by {D}.
  {G}oldfeld).
\newblock {\em Invent. Math.}, 44(2):129--162, 1978.

\bibitem{mazurSD}
B.~Mazur and P.~Swinnerton-Dyer.
\newblock Arithmetic of {W}eil curves.
\newblock {\em Invent. Math.}, 25:1--61, 1974.

\bibitem{qiuzhang}
C.~Qiu and W.~Zhang.
\newblock {I}njectivity of the {A}bel-{J}acobi map and {G}ross-{K}udla-{S}choen
  cycles.
\newblock In preparation.

\bibitem{ribet75}
K.~A. Ribet.
\newblock Endomorphisms of semi-stable abelian varieties over number fields.
\newblock {\em Ann. of Math. (2)}, 101:555--562, 1975.

\bibitem{setzer}
B.~Setzer.
\newblock Elliptic curves of prime conductor.
\newblock {\em J. London Math. Soc. (2)}, 10:367--378, 1975.

\bibitem{shimura58}
G.~Shimura.
\newblock Correspondances modulaires et les fonctions {$\zeta $} de courbes
  alg\'{e}briques.
\newblock {\em J. Math. Soc. Japan}, 10:1--28, 1958.

\bibitem{yzz}
X.~Yuan, S.~Zhang, and W.~Zhang.
\newblock {T}riple product {$L$}-series and {G}ross-{K}udla-{S}choen cycles.
\newblock Preprint, 2012.
  \url{http://math.mit.edu/~wz2113/math/online/triple.pdf}.

\bibitem{zhang10}
S.~Zhang.
\newblock Arithmetic of {S}himura curves.
\newblock {\em Sci. China Math.}, 53(3):573--592, 2010.

\end{thebibliography}

\end{document}